\documentclass[11pt,reqno]{amsart}
\usepackage{amssymb}
\usepackage[T1]{fontenc}
\usepackage[dvipsnames]{xcolor}
\usepackage{palatino}
\input amssym.def
\usepackage{amsmath, amsfonts, enumitem}
\usepackage{amscd, mathtools}
\usepackage[mathscr]{eucal}
\usepackage{palatino}
\newfont{\cyrr}{wncyr10}

\newcommand{\N}{{\mathfrak N}}
\newcommand{\Z}{{\mathbb Z}}
\newcommand{\Q}{{\mathbb Q}}
\newcommand{\R}{{\mathbb R}}
\newcommand{\C}{{\mathbb C}}

\newcommand{\F}{{\mathbf F}}
\newcommand{\K}{{\mathbf K}}
\newcommand{\bN}{{\mathbb N}}

\newcommand{\f}{{\mathbf{f}}}
\newcommand{\kk}{{\mathbf{k}}}

\newcommand{\thmref}[1]{Theorem~\ref{#1}}
\newtheorem{thm}{Theorem}
\newtheorem*{thm*}{Theorem}

\newtheorem{lem}[thm]{Lemma}
\newtheorem{cor}[thm]{Corollary}

\newtheorem{rmk}{Remark}[section] 
\newtheorem*{defn}{Definition}

\newcommand{\lemref}[1]{Lemma~\ref{#1}}

\newcommand{\rmkref}[1]{Remark~\ref{#1}}

\newcommand{\m}{{\mathfrak m}}

\newcommand{\n}{{\mathfrak{n}}}
\newcommand{\cO}{\mathcal{O}}
\newcommand{\bH}{\mathbb{H}}
\newcommand{\ka}{\kappa}
\newcommand{\GL}{\text{GL}}
\newcommand{\tr}{{\rm tr}}

\begin{document}

\title[Quotients of derivatives of $L$-functions]
{On quotients of derivatives of $L$-functions
inside the critical strip}

\author{Rashi Lunia}

\address{Rashi Lunia  \\ \newline
	The Institute of Mathematical Sciences, 
	A CI of Homi Bhabha National Institute, 
	C.I.T Campus, Taramani, Chennai 600 113, India.}

\email{rashisl@imsc.res.in}



\begin{abstract}
In 2011, Gun, Murty and Rath studied non-vanishing and transcendental nature
of special values of a varying class of $L$-functions and their derivatives. This led
to a number of works by several authors in different set-ups including studying higher
derivatives. However, all these works were focused around the central point of the
critical strip. In this article, we extend the study to arbitrary points in
the critical strip. 
\end{abstract}

\maketitle   

\section{\large{Introduction and Preliminaries}}
Special values of $L$-functions have been the focus of study since
the time of Euler.  In recent times, Gun, Murty and Rath \cite{GMR} studied 
non-vanishing of modular $L$-functions  along with their derivatives   
and deduced information about the algebraic nature of their 
values at the center of the critical strip. 

Tanabe \cite{NT} extended their result to Hilbert modular forms and 
later Kumar \cite{NK} proved similar results for half- integer weight modular forms. 
In 2016,  Murty and Tanabe \cite{MT} showed
 that vanishing of the derivative of certain Artin $L$-function at $1/2$ 
 is related to the transcendence of $e^{\gamma}$ (see also \cite{NeK}). 
However, all these works were focused around the central point of the
critical strip. 

In this article, we investigate derivatives of quotients of $L$-functions
at arbitrary points inside the critical strip. Further, we derive $\overline{\Q}$-linear 
independence results for such $L$-values. We also deduce transcendence 
degree of the field generated by these values. For results of different nature 
for rational points in the critical strip, see \cite{GMR2}. 

\subsection{Notations}
Throughout this paper, we say that a finite set of integers $S$ has 
property $\bf{A}$ if for every $n \in S$, there exists a prime $p$ such that 
$p | n$ but $p \nmid m$ for any other $m \in S$. Note that any finite subset of 
prime numbers satisfies property $\bf A$.
We shall use the notation $[x]$ to denote the greatest 
integer $n\leq x$.
We shall use $\Z_{\ge 1}$,  $\bN$ and $\Z_{\le 0}$ to denote the set of positive 
integers, non-negative integers and non-positive integers respectively.

\begin{defn}
The transcendence degree of a field $F$ over $\Q$ is the cardinality of a maximal 
algebraically independent subset of $F$ over $\Q$. 
\end{defn}
For a set $B$, we use $\tr_{\Q}(B)$ to denote the transcendence degree of the 
field $\Q(B)$ over $\Q$. 

\subsection{Transcendental Pre-requisites}
We now recall a few results from transcendental number theory.
\begin{thm*}[Lindemann \cite{FL}]
If $\alpha \neq 0, 1$ is an algebraic number, then $\log{\alpha}$ is transcendental. 
\end{thm*}

Throughout this paper, without loss of generality, we consider the principal 
branch of logarithm. In a seminal work on the theory of linear forms in logarithms, 
Alan Baker proved the following.
 
\begin{thm*}[Baker \cite{AB}]\label{Baker}
If $\alpha_1, \ldots, \alpha_n$ are non-zero algebraic numbers 
such that $\log \alpha_1, \ldots, \log \alpha_n$ are linearly independent over $\Q$, then 
$1, \log \alpha_1, \ldots, \log \alpha_n$ are linearly independent over $\overline{\Q}$.
\end{thm*}

\begin{subsection}{Polygamma function}
The gamma function is defined as 
$$ 
\Gamma (z)=\int_{0}^{\infty }t^{z-1}e^{-t}\,dt,\ \qquad z \in \C,  ~~\Re(z)>0,
$$
and it satisfies the functional equation $\Gamma(z+1)=z\Gamma(z)$. 
Let $\psi$ denote the digamma function which is the logarithmic 
derivative of the gamma function. For $ z \in \C\backslash \Z_{\leq 0}. $ we have
$$
-\psi(z)= \gamma + \frac{1}{z}+ \sum_{n=1}^{\infty}\left(\frac{1}{n+z}-\frac{1}{n}\right),
$$
and $\psi(1)=-\gamma$, where $\gamma$ is Euler's constant. When $x$ is 
real, $\psi(x)$ is strictly increasing on $(0, \infty)$.
The digamma function satisfies the following relations \cite{AS}:
\begin{enumerate}[label=(\roman*)]
\item Recurrence formula: $\psi(z+1)=\psi(z)+\frac{1}{z}.$\label{Rec}
\item Reflection principle: $\psi(1-z)=\psi(z)+\pi\cot{\pi z}.$
\item Duplication formula: $2\psi(2z)=\psi(z)+\psi(z+1/2)+2\log{2}$.
\end{enumerate}

The higher derivatives of the digamma function are known as 
polygamma functions. They are defined as 
\begin{equation}\label{zetagamma}
\psi^{(m)}(z)=(-1)^{m+1}m!\sum_{k=0}^{\infty}\frac{1}{(k+z)^{m+1}}
\end{equation}
and are holomorphic on $\C\backslash \Z_{\leq 0}.$
Analogous to the above, polygamma functions satisfy the following relations 
(see eqns 6.4.6 and 6.4.7 of \cite{AS}):  For $m \in \Z_{\ge 1}$,
 we have
\begin{enumerate}[label=(\roman*)]
\item 
Recurrence relation:
\begin{equation}\label{Gamma}
\psi^{(m)}(z+t)=\psi^{(m)}(z)+(-1)^mm!\sum_{j=0}^{t-1}\frac{1}{(z+j)^{m+1}}, 
\quad \text{for  }  t \in \Z_{ \ge 1}.  
\end{equation}
\item Reflection principle:
\begin{equation}\label{Reflection}
\psi^{(m)}(1-z)+(-1)^{m+1}\psi^{(m)}(z)=(-1)^m\pi\frac{d^m}{dz^m}(\cot{\pi z}).
\end{equation}
\item Duplication formula: 
\begin{equation}\label{Duplication}
 2^{m+1}\psi^{(m)}(2z)=\psi^{(m)}(z)+\psi^{(m)}(z+ 1/2).
\end{equation}
\end{enumerate}
where the last equation is obtained by taking $m$-th derivative 
of the duplication formula for digamma function.
Further, taking $z=1/2$ in \eqref{Duplication}, for $ m\ge 1$, we get
\begin{equation}\label{psi12}
 \psi^{(m)}(1/2)=(2^{m+1}-1)\psi^{(m)}(1).
\end{equation}

\begin{thm*}[Gun-Murty-Rath \cite{GMR3}]\label{dim}
There exists an integer $q_0>1$ such that for any integer 
$q$ coprime to $q_0$, the $\Q$-vector space spanned by the set
$$
\{\psi(a/q)| 1\leq a \leq q, (a,q)=1\}
$$
has dimension $\phi(q)$.
\end{thm*}

\begin{rmk}\label{dimension}
It follows from the above theorem that for all integers $q\geq 2$ coprime to $q_0$, 
the $\Q$-vector space spanned by the set 
$$
\left\{\psi\left(\frac{a}{q}\right)+\psi\left(1-\frac{a}{q}\right)\Bigg|~ 
1\leq a \leq q, (a,q)=1\right\}
$$
has dimension $\phi(q)/2$.
\end{rmk}

\end{subsection}

\begin{subsection}{An elementary lemma}
For the sake of completion, we prove the following lemma.
\begin{lem}\label{LI}
Let $V$ be a vector space over a field $\K$ and $\{v_1, \ldots v_n\}$ be a
 linearly independent set in $V$. Then for any $w \in V$ and $r_i \in \K$, the set 
 $\{v_1-r_1w, \ldots, v_n-r_nw\}$ spans a subspace of $V$ of dimension at least $n-1$.
\end{lem}

\begin{proof}
If the set $\{v_1, v_2, \ldots, v_n, w\}$ is linearly independent over $\K$, 
then the result is trivially true. Without loss of generality, let
$w=\sum_{i=1}^n\beta_iv_i$, for some $\beta_i \in \K$ with $\beta_1 \neq 0$.
We shall show that the set  $\{v_2-r_2 w, \ldots, v_n-r_n w\}$ is 
linearly independent over $\K$ for any $r_i \in \K$.
Let $\alpha_i \in \K$ be such that 
$$
0=\sum_{i=2}^n\alpha_i(v_i-r_iw)=
\sum_{i=2}^n\alpha_iv_i-\sum_{i=2}^n\alpha_ir_i\left(\sum_{j=1}^n\beta_jv_j\right).
$$
Since $\beta_1 \neq 0$ and $\{v_1, v_2, \ldots, v_n\}$ is linearly independent over 
$\K$, we get the result.
\end{proof}
\end{subsection}

\section{\large{ Automorphic $L$- functions}}
Let $d \geq 1$ be an integer and $f$ be a cuspidal automorphic form on $\GL(d)$
which is a Hecke eigenform with eigen values
$\lambda_f(n)$. Let  
$$
 L(f,s)=\sum_{n\geq 1}\frac{\lambda_f(n)}{n^s}, \qquad \Re(s)>1,
 $$ 
be the automorphic $L$-function of degree $d$ associated to $f$. 
The completed $L$-function associated to $f$ is defined as 
\begin{equation}\label{CLFGLD}
\Lambda(f, s)= N^{s/2}\pi^{-ds/2}\prod_{j=1}^{d}\Gamma\left(\frac{s+\ka_j}{2}\right)L(f, s),
\end{equation} 
where the integer $N\geq 1$ is the conductor of $L(f, s)$ and the  
complex numbers $\ka_j$ are the local parameters of $f$ at infinity 
with $\Re(\ka_j)>-1$. Luo, Rudnick and Sarnak \cite{LRS} showed 
that for any automorphic cusp form, $\Re(\ka_j)>-c$ where $c= \frac{1}{2}-\frac{1}{d^2+1}$.
The completed $L$-function $\Lambda(f,s)$ admits a meromorphic 
continuation to $\C$ with atmost  poles at $s=0$ and $s=1$ and it 
satisfies the functional equation 
\begin{equation}\label{FunGLD}
\Lambda(f,s)= \epsilon(f)\Lambda(\overline{f},1-s)
\end{equation}
where $\overline{f}$ is the dual of $f$ and $\epsilon(f)$ is a complex 
number with absolute value $1$, called the root number of $f$.
If $\overline{f}=f$, then $f$ is said to be self-dual. In this case, 
$\lambda_f(n) \in \R$ and $\epsilon(f)=\pm 1$. 

For integers $d,N \geq 1$ and a point $s_0\in \C$ with $\Re(s_0)\in (0,1)$, 
we define $E(d,\ka,N,s_0)$ to be the set of all cuspidal, self-dual 
automorphic Hecke eigenforms $f$ of degree $d$ and conductor $N$ such that 
their local parameters at infinity are $\ka=\{\ka_1, \ldots, \ka_d\}$ 
and $L(f,s_0)\neq 0$. In this set-up, we have the following results.

\begin{thm}
Let $s_0\in \C$. Then at most one element of the set
$$
\left\{\frac{L'(f,s_0)}{L(f,s_0)}+\frac{L'(f,1-s_0)}{L(f,1-s_0)} 
~\Big|~ 
f \in E(d,\ka,N,s_0), N \geq 1\right\}
$$
is algebraic. Furthermore, if $s_0$ and $\ka_j$ are real and satisfy
$$
s_0+\ka_j \ge 1 \quad {\rm and } \quad 1-s_0+\ka_j \ge 1 \quad {\rm for~ all~ }j
$$
and $N^{1/d}>4\pi e^{\gamma}$, we have
$$
\frac{L'(f,s_0)}{L(f,s_0)}+\frac{L'(f,1-s_0)}{L(f,1-s_0)} \neq 0.
$$
\end{thm}

\begin{proof}
Let $f \in E(d,\ka,N,s_0)$. From \eqref{CLFGLD} and \eqref{FunGLD} we have 
\begin{equation*}
N^{s/2}\pi^{-ds/2}\prod_{j=1}^{d}\Gamma\left(\frac{s+\ka_j}{2}\right)L(f,s)
=
\epsilon(f)
N^{(1-s)/2}\pi^{-d(1-s)/2}\prod_{j=1}^{d}\Gamma\left(\frac{(1-s)+\ka_j}{2}\right)L(f,1-s).
\end{equation*}
Taking the logarithmic derivative, we get 
\begin{equation*}
\log{N} + \frac{1}{2}\sum_{j=1}^d\psi\left(\frac{s+\ka_j}{2}\right)
+\frac{L'(f,s)}{L(f,s)}=d\log{\pi}
-\frac{1}{2}\sum_{j=1}^d\psi\left(\frac{1-s+\ka_j}{2}\right)
-\frac{L'(f,1-s)}{L(f,1-s)}.
\end{equation*}
Since by assumption, $L(f,s_0)\neq 0$, putting $s=s_0$ in the above equation, we get
\begin{equation}\label{glds0}
\frac{L'(f,s_0)}{L(f,s_0)} +\frac{L'(f,1-s_0)}{L(f,1-s_0)}
=-\frac{1}{2}\sum_{j=1}^d\left[\psi\left(\frac{s_0+\ka_j}{2}\right)
+\psi\left(\frac{1-s_0+\ka_j}{2}\right)\right]+d\log{\pi}-\log{N}.
\end{equation}
Let $f_1 \in E(d,\ka,N_1,s_0)$ and $f_2\in E(d,\ka,N_2,s_0)$ be such 
that their corresponding values are algebraic. By \eqref{glds0}, 
it is enough to consider the case when $N_1 \neq N_2$. In this case, 
$$
\frac{L'(f_1,s_0)}{L(f_1,s_0)}+\frac{L'(f_1,1-s_0)}{L(f,1-s_0)}
- \frac{L'(f_2,s_0)}{L(f_2,s_0)}-\frac{L'(f_2,1-s_0)}{L(f_2,1-s_0)} = \log\frac{N_2}{N_1}
$$
is algebraic. This gives a contradiction by Lindemann's theorem,
and we get the first part of the result.
For $s_0$ and $\ka_j$ satisfying the conditions in the second part
of the theorem, using the fact that $\psi(x)$ is strictly increasing
on $(0,\infty)$, we get
$$
-\psi\left(\frac{s_0+\ka_j}{2}\right)< 
-\psi\left(\frac12 \right) 
\quad {\rm and } \quad
-\psi\left(\frac{1-s_0+\ka_j}{2}\right)< 
-\psi\left(\frac12 \right).
$$
Therefore  using \eqref{glds0} and putting $-\psi(1/2)=\gamma+2\log{2}$, we get 
$$
\frac{L'(f,s_0)}{L(f,s_0)}+\frac{L'(f,1-s_0)}{L(f,1-s_0)} < 2d\log{2}
+d\gamma +d\log{\pi}-\log{N}<0
$$
whenever $N^{1/d}>4\pi e^{\gamma}$.
\end{proof}

In particular, we have

\begin{cor} 
Atmost one element of the set
$$
\left\{ \frac{L'(f,\frac12)}{L(f,\frac12)} 
~\Big|~
 f\in E\left(d,\ka,N,\frac12\right), N \geq 1 \right\}
$$
is algebraic. Furthermore, if $\ka_j$ is real with $\ka_j \ge \frac{1}{2}$,
for all $j$ and $N^{1/d}>4\pi e^{\gamma}$ then $L'(f,\frac{1}{2})\neq 0$.
\end{cor}

\begin{thm}
Let $s_0\in \C$ and let $J$ be a finite set of natural numbers $N\geq 1$ such that 
$E(d,\ka,N, s_0)$ is non-empty. Further, assume that $J$ 
has property $\bf{A}$. Then the set
$$
\left\{ \frac{L'(f,s_0)}{L(f,s_0)} +\frac{L'(f,1-s_0)}{L(f,1-s_0)}
~\Big|~
 f\in E(d,\ka,N,s_0), N \in J \right\}
$$
spans a vector space over $\overline{\Q}$ of dimension at least $|J|-1$.
\end{thm}
\begin{proof}
Since the set $J$ has property $\bf{A}$, the set $\{\log{N}\big| N \in J\}$ 
is linearly independent over $\Q$.
Hence by Baker's theorem, it is a linearly independent 
set over $\overline{\Q}$ of cardinality $|J|$. Taking 
$$
w
= - \frac{1}{2}\sum_{j=1}^d\left[\psi\left(\frac{s_0+\ka_j}{2}\right)
+\psi\left(\frac{1-s_0+\ka_j}{2}\right)\right] + d\log{\pi},
$$
and using \eqref{glds0} and \lemref{LI}, we get the desired result.
\end{proof}

\section{\large{Integer and Half-integer weight Modular Forms}}
Let $N \geq 1$ be an integer and $0 \neq k\in \frac{1}{2}\mathbb{N}$. 
Let $f$ be a cusp form of weight $k$ for $\Gamma_0(N)$ (when $k$ 
is not an integer, we assume that $N\equiv 0 \pmod{4}$) with trivial 
nebentypus having Fourier expansion at infinity
$$
f(z)= \sum_{n=1}^{\infty}a_f(n)q^n.
$$
Then the $L$-series attached to $f$ is defined by
$$
L(f, s)= \sum_{n=1}^{\infty}\frac{a_f(n)}{n^s}
$$
which converges  for $\Re(s)\gg 1$. 
Throughout this section, $D$ will denote a fundamental discriminant. 
For $(D,N)=1$, 
the $L$-series of $f$ twisted with the quadratic character associated 
to $D$ is given by 
$$
L(f, D, s) = \sum_{n=1}^{\infty}\left(\frac{D}{n}\right)\frac{a_f(n)}{n^s}
$$
which converges for $\Re(s)\gg 1$ and has a holomorphic 
continuation to $\C$. Consider the completed $L$-function
\begin{equation}\label{FunEMF1}
\Lambda(f, D, s)=(2\pi)^{-s}(ND^2)^{s/2}\Gamma(s)L(f,D,s).
\end{equation}
Then $\Lambda(f,D,s)$ has a meromorphic continuation to $\C$ and satisfies 
the functional equation 
\begin{equation}\label{FunEMF}
\Lambda(f,D,s)=i^k\left(\frac{D}{-N}\right)\Lambda(f|_k{W_N},D,k-s)
\end{equation} 
where
$f|_k{ W_N}(z)= N^{-k/2}z^{-k}f(-1/Nz)$ (see Section 1.5 of \cite{DB} and \cite{GS1}).

For $s_0 \in \C$, let $E(N,D,k, s_0) $ be the set of all normalized cuspidal Hecke 
eigenforms $f$ of weight $k$ and level $N$ such that $L(f,D,s_0) \neq 0$. 
For integers $k, N \ge 1$, it is known that
there are infinitely many fundamental discriminants $D$ such that $(D,N)=1$ 
and $E(N, D, k, k/2)$ is non-empty (see \cite{BFH, MM, JW1, JW2}). 
On the other hand, fixing a fundamental discriminant  $D$, there are
infinitely many integers $N$ such that $E(N,D,k, k/2)$ is non-empty (see \cite{AA, WD, KM, JV} ). 

Moreover, it is known by the works of several authors \cite{WK, AR, RS, MS}
that for $k \gg 1$ and $D$ a fundamental discriminant, there are infinitely 
many integers $N$ such that $\cap_{s_0\in S}E(N,D,k, s_0)$ is non-empty for certain 
sets $S$ consisting of points in the critical strip.
In this set-up, we have the following results.

\begin{thm}
Let $s_0\in \C$ and $k \geq 1$. Then at most one element of the set
$$
\left\{\frac{L'(f,D,s_0)}{L(f,D,s_0)}+
\frac{L'(f,D,k-s_0)}{L(f,D,k-s_0)}
~\Big|~ 
 f \in E(N,D,k,s_0), ~ND^2\ge 1 {\rm ~with ~} (N,D)=1 \right\}
$$
is algebraic. Furthermore, for $k \geq 3$ and $s_0 = k/2 + a/b \in \Q$ with 
$0 \leq a/b \leq k/2-1$ if $ND^2 >  4\pi^2e^{2\gamma}$
then $f \in  E(N,D,k, s_0)$, implies
$$
\frac{L'(f,D,s_0)}{L(f,D,s_0)}+\frac{L'(f,D,k-s_0)}{L(f,D,k-s_0)}\neq 0.
$$
\end{thm} 
\begin{rmk}
Same result holds if we replace $k \geq 3$ by $k \geq 5$ and 
$ND^2 >  4\pi^2e^{2\gamma}$ by $ND^2 >  4 \pi^2 e^{2\gamma - 2 }$.
\end{rmk}
\begin{proof}
Let $f \in E(N,D,k, s_0)$. Using \eqref{FunEMF} and 
\eqref{FunEMF1}, and taking logarithmic derivative, we get
\begin{equation}\label{relation}
2\log{2\pi}-\log{ND^2}-\psi(k-s)
-
\frac{L'(f, D ,k-s)}{L(f, D, k-s)}=\psi(s) + \frac{L'(f, D, s)}{L(f, D, s)}
\end{equation}
Putting $s=s_0$, in the above equation, we get
\begin{equation}\label{fracab}
\frac{L'(f,D,s_0)}{L(f,D,s_0)}+\frac{L'(f,D,k-s_0)}{L(f,D,k-s_0)}
=2\log{2\pi}-\psi(s_0)-\psi(k-s_0)-\log{N}-2\log{D}.
\end{equation}
Let $f_1\in E(N_1,D_1,k,s_0)$ and $f_2\in E(N_2,D_2,k,s_0)$ be 
such that their corresponding values are distinct and algebraic. 
Then by \eqref{fracab}, 
$$
\frac{L'(f_1,D_1,s_0)}{L(f_1,D_1,s_0)}+\frac{L'(f_1,D_1,k-s_0)}{L(f_1,D_1,k-s_0)} 
- \frac{L'(f_2,D_2,s_0)}{L(f_2,D_2,s_0)}-\frac{L'(f_2,D_2,k-s_0)}{L(f_2,D_2,k-s_0)}
=\log{\frac{N_2D_2^2}{N_1D_1^2}}
$$
is also algebraic, which is a contradiction to Lindemann's theorem. 
For the second part of the theorem, the assumptions
on $s_0$ and $k$ imply $k/2\pm a/b \ge 1$. 
Putting $s_0=  k/2 + a/b $ in \eqref{fracab}, we get 
\begin{equation*}
\frac{L'(f,D,s_0)}{L(f,D,s_0)}+\frac{L'(f,D,k-s_0)}{L(f,D,k-s_0)}
\leq  2\log{2\pi} - \log{ND^2} + 2\gamma<0
\end{equation*}
whenever $ND^2>4\pi^2e^{2\gamma}$.
\end{proof}

The case $s_0=k/2$, has been studied in \cite{GMR,NK}. This result can 
be considered a generalisation of their result. The result for half-integer 
weight modular forms of weight $k$ and level $N$, at $s_0=k/2$  
was mentioned in \cite{RS} but the contribution 
of level $N$ was missing in their expression. 
We therefore state the result here.

\begin{cor}
Let $k \in \frac{1}{2}+\mathbb{N}$ and $f\in \mathcal{S}_{k}(\Gamma_0(N))$ 
be a Hecke eigen cusp form such that 
$L(f,k/2)\neq 0$. Then, we get 
$$
\frac{L'(f,k/2)}{L(f,k/2)}= \log\pi-\psi(k/2)-\frac{1}{2}\log \frac{N}{4}.
$$
Furthermore, $L'(f,k/2)\neq 0$.
\end{cor}
\begin{proof} The proof of the first part follows by 
taking $s_0=k/2$ and $D=1$ in \eqref{fracab}.
Non-vanishing of $L'(f, k/2)$ when $N=4$ was shown in \cite{NK}. 
Therefore, we assume $N>4$.
For $k \in \frac{1}{2} + 2 \bN$, $k>2$,
$$
\frac{L'(f, k/2)}{L(f, k/2)} 
= 
\log{\pi}-\psi(1/4)
-
\sum_{n=0}^{\left[\frac{k}{2}\right]-1}\frac{4}{4n+1}
-\frac{1}{2}\log{\frac{N}{4}}.
$$
It is known that
$-\psi(1/4)=\gamma+3\log{2}+\frac{\pi}{2}$. 
For $k > 6,$
$$
\log \pi-\psi(1/4)<5.37223<\sum_{n= 0}^{\left[\frac{k}{2}\right]-1}\frac{4}{4n+1}
+\frac{1}{2}\log{2}
$$ 
and therefore $L'(f, k/2)\neq 0$.
When $k \in \frac{3}{2}+2\mathbb{N}$, $k>2$, 
$$
\frac{L'(f,k/2)}{L(f,k/2)}=\log{\pi}-\psi(3/4)
-\sum_{n= 0}^{\left[\frac{k}{2}\right]-1}\frac{4}{4n+3}
-\frac{1}{2}\log{\frac{N}{4}}.
$$
Using the reflection principle, we get $\psi(3/4)=\psi(1/4)+\pi$. 
Now, for $k > 5$,
$$
\log{\pi}-\psi(3/4)<2.231<\sum_{n= 0}^{\left[\frac{k}{2}\right]-1}
\frac{4}{4n+3}+\frac{1}{2}\log{2},
$$
which implies $L'(f, k/2)\neq 0$. The remaining cases follow by similar computations.
\end{proof}

\begin{thm}\label{main2}
Let $ N \geq 1$ be an integer and $k \in \frac12\mathbb{N}$. 
Let $s_0 \in \C$ and $J_N$ be a finite set consisting of fundamental 
discriminants $D$ such that $(D,N)=1$ and $ E(N,D,k, s_0)$ is 
non-empty. Further assume that $J_N$ has property
$\bf{A}$. Then the set
$$
\left\{ \frac{L'(f,D,s_0)}{L(f,D,s_0)} + \frac{L'(f, D, k-s_0)}{L(f, D, k-s_0)}
~\Big|~
 f\in E(N, D, k, s_0), D \in J_N \right\}
 $$ 
 spans a vector space over $\overline{\Q}$ of dimension at least $|J_N|-1$.
\end{thm}

\begin{thm}\label{main3}
Let $k \in \frac12\mathbb{N}$ and $s_0 \in \C$.  
For a fundamental discriminant $D$, let $J_D$ be a finite set
of natural numbers $N \geq 1$ 
such that $ (D,N)=1$ and $E(N,D,k, s_0)$ is non-empty. 
Further assume that $J_D$ has property $\bf{A}$. Then the set
$$
\left\{ \frac{L'(f,D,s_0)}{L(f,D,s_0)} +\frac{L'(f, D, k-s_0)}{L(f, D, k-s_0)}
~\Big|~ 
f\in E(N,D,k,s_0), N \in J_D \right\}
$$ 
spans a vector space over $\overline{\Q}$ of dimension at least $|J_D|-1$.
\end{thm}

\begin{proof}[Proof of \thmref{main2} and \thmref{main3}]
Using unique factorization of primes, the sets 
$\{\log{D}\big| D \in J_N\}$ and $\{\log{N}\big| N \in J_D\}$ 
are linearly independent over $\Q$. Hence by Baker's theorem, 
they are linearly independent sets over $\overline{\Q}$ 
of cardinality $|J_N|$ and $|J_D|$ respectively. Choosing
$w=2\log{2\pi}-\psi(s_0)-\psi(k-s_0)-\log{N}$ and 
$w=2\log{2\pi}-\psi(s_0)-\psi(k-s_0)-2\log{D}$ in the respective 
cases and using \eqref{fracab} and \lemref{LI} we get the results.
\end{proof}

In particular, we have
\begin{cor}
Let $k \in \frac12\mathbb{N}$ and $D$ be a fundamental discriminant.
Let $J_D$ be a finite set of natural
numbers $N \geq 1$ 
such that $(D,N)=1$ and $E(N,D,k, k/2)$ is non-empty. 
Further assume that $J_D$ has property $\bf{A}$. Then the set
$$
\left\{ \frac{L'(f,D,k/2)}{L(f,D,k/2)} 
~\Big|~ 
f\in E(N,D,k,k/2), N \in J_D \right\}
$$ 
spans a vector space over $\overline{\Q}$ of dimension at least $|J_D|-1$.
\end{cor}

Instead of investigating the twists of modular $L$-values at arbitrary 
points inside the critical strip, if we specialize at rational points $\frac{a}{q}$
 with $(a, q)=1$, we have the following theorem.
\begin{thm}\label{coprime}
Let $k \geq 2$. Then there exists an integer $q_0$ such that for any
integer $q \ge 7$ coprime to $q_0$ and 
$f \in \cap_{\substack{1\leq a \leq q\\ (a,q)=1}}E(N,D,k,a/q)$, the 
$\Q$-vector space spanned by the set
$$
\left\{ \frac{L'(f,D, a/q)}{L(f, D, a/q)} 
+\frac{L'(f, D, k-a/q)}{L(f, D, k-a/q)}
~\Big|~
 1 \leq a  <q/2, (a,q)=1 \right\}
$$
has dimension at least 
$\frac{\phi(q)}{2}-2$. 
\end{thm}

\begin{proof}
Putting $s_0=a/q$ in \eqref{fracab}, we get
$$
\frac{L'(f, D,a/q)}{L(f,D, a/q)} +\frac{L'(f, D, k-a/q)}{L(f, D, k-a/q)}
=\log\frac{4\pi^2}{ND^2}-\psi\left(\frac{a}{q}\right)-\psi\left(1-\frac{a}{q}\right)
-r(k,a,q),
$$
where $\displaystyle r(k,a,q)= \sum_{j=1}^{k-1}\frac{1}{j-a/q} \in \Q$. 
The result now follows by \rmkref{dimension} and \lemref{LI}.
\end{proof}

\subsection{Higher derivatives of $L$-functions associated to integer  and
half-integer weight modular forms}

For any integer $m\geq 1$, taking the $m$-th derivative of 
\eqref{relation}, we get  
\begin{equation}\label{highMF}
\left(\frac{L'(f,D,s)}{L(f,D,s)}\right)^{(m)}
+
(-1)^m\left(\frac{L'(f,D,k-s)}{L(f,D,k-s)}\right)^{(m)}
=
- \psi^{(m)}(s) + (-1)^{m+1 } \psi^{(m)}(k-s).
\end{equation}
In this set-up, we have the following theorems.

\begin{thm}\label{thm11}
For $M \in \mathbb{N}$, and $1 \leq \beta \leq 4$, consider the sets
\begin{equation*}
\mathcal{L}_{\beta,M}
=
\left\{
\frac{L^{(m)}(f,D,k/2)}{L(f,D,k/2)} 
\middle\vert 
\begin{array}{c} 
N\in \Z_{\ge 1},~ (D,N)=1,~ 2k\equiv \beta\!\!\! \pmod{4}\\
 f \in E(N, D, k, k/2), 1 \leq m \leq M 
\end{array} 
\right\}.
\end{equation*}
Then we have  
\begin{equation*} 
\tr_{\Q}(\mathcal{L}_{\beta, M}) 
~\geq~ 
\tr_{\Q}\left(\left\{\psi^{(2m)}
\left(\frac{\beta}{4}\right)
~\Big| ~
1 \leq m \leq \left[\frac{M-1}{2}\right]\right\}\right).
\end{equation*}
\end{thm}

\begin{proof}
For $1 \leq \beta \leq 4$ and $k \in \frac12\Z_{\ge 1}$, 
let $f \in E(N, D, k, k/2)$. Putting $s=k/2$ in
\eqref{highMF}, we get 
\begin{equation*}
\left(\frac{L'(f,D,s)}{L(f,D,s)}\right)^{(2m)}_{|s=k/2}=-\psi^{(2m)}\left(\frac{k}{2}\right)
\quad\text{for } m \geq 1.
\end{equation*}
Applying \eqref{Gamma}, we have
\begin{equation}\label{psik2}
\psi^{(2m)}\left(\frac{k}{2}\right)=\psi^{(2m)}\left(\frac{\beta}{4}\right)
+
4^{2m+1}(2m)!\sideset{}{'}\sum_{j=0}^{\left[\frac{k}{2}\right]-1}(4j
+
\beta)^{-2m-1}-\delta_k\frac{2^{2m+1}(2m)!}{k^{2m+1}}
\end{equation}
where $\delta_k=1$ if and only if $k \in 2\Z_{\ge 1}$. The notation $\sum'$ 
denotes that the sum is $0$ when $k < 2$. Let 
\begin{align*}
\tilde{\mathcal{L}}_{\beta,M} = 
\left\{\left(\frac{L'(f,D,s)}{L(f,D,s)}\right)^{(m)}|_{s=k/2}
\middle\vert 
\begin{array}{c} 
N\in \Z_{\ge 1},~ (D,N)=1,~ 2k\equiv \beta \!\!\! \pmod{4}\\
 f \in E(N, D, k, k/2), 0 \leq m \leq M -1
\end{array} 
\right\}.
\end{align*}
Using induction we see that for any  $j \in \Z_{\ge 1}$,  
$\displaystyle\frac{L^{(j)}(f,D,s)}{L(f,D,s)}$ can be expressed as a 
polynomial in $\displaystyle\frac{L'(f,D,s)}{L(f,D,s)}, 
\Big(\frac{L'(f,D,s)}{L(f,D,s)}\Big)^{(1)}, \ldots , 
\Big(\frac{L'(f,D,s)}{L(f,D,s)}\Big)^{(j-1)}$ 
with coefficients in $\Z$ and conversely 
for any ${j \in \bN}$, 
$\displaystyle\Big(\frac{L'(f,D,s)}{L(f,D,s)}\Big)^{(j)}$
can be expressed as a polynomial in
$\displaystyle\frac{L'(f,D,s)}{L(f,D,s)},
\frac{L^{(2)}(f,D,s)}{L(f,D,s)}, \ldots,$ $\displaystyle \frac{L^{(j+1)}(f,D,s)}{L(f,D,s)}$
with coefficients in $\Z$.
This implies  
$
\tr_{\Q}(\mathcal{L}_{\beta,M})=\tr_{\Q}(\tilde{\mathcal{L}}_{\beta,M})
$,
which completes the proof.
\end{proof}
\begin{rmk}\label{rmk}
It follows from \eqref{zetagamma} that 
$\psi^{2m}(1)=-(2m)!\zeta(2m+1)$.
When $ \beta = 2$ or $ 4$, we have 
\begin{equation*} 
\tr_{\Q}(\mathcal{L}_{\beta,M}) 
~\geq~ 
\tr_{\Q}\Big(\Big\{\psi^{(2m)}(1)~\Big| ~1 \leq m \leq \Big[\frac{M-1}{2}
\Big]\Big\}\Big)
=
\tr_{\Q}\Big(\Big\{\zeta(2m+1)~\Big| ~1 \leq m \leq \Big[\frac{M-1}{2}
\Big]\Big\}\Big).
\end{equation*}
The case $\beta = 4$ is a consequence of \thmref{thm11},
while the case $\beta =2 $ then follows by using \eqref{psi12}.
A folklore conjecture states that $\zeta(3), \zeta(5), \ldots $ are 
algebraically independent over $\Q$ which then implies
$$
\tr_{\Q}(\mathcal{L}_{\beta,M}) \geq \Big[\frac{M-1}{2}\Big].
$$
\end{rmk}

\begin{thm}
For $s_0 \in \Q$, consider the set
$$
\mathcal{S}(s_0) =
 \{ f \in E(N, D, k, s_0) ~~|~ N \geq 1, (D,N)=1, k \in \frac12 \bN \}.
$$
For any $m \geq 1$, if at least one element of the set 
$$
\Big\{\Big(\Big(\frac{L'(f,D,s)}{L(f, D, s)
}\Big)^{(m)}
+
(-1)^m\Big(\frac{L'(f, D, k-s)}{L(f, D, k-s)}
\Big)^{(m)}\Big)
|_{s=s_0}
~\Big|~
f \in S(s_0) \Big\}
$$
is algebraic, then all the elements of this set are algebraic.
\end{thm}

\begin{proof}
Let $f_1,f_2\in \mathcal{S}(s_0)$. By \eqref{highMF}, we know that 
the result is true if $k_1=k_2$. For $k_1\neq k_2$, again 
using \eqref{highMF}, we note that the difference of their 
corresponding values is 
$$(-1)^m(\psi^{(m)}(k_1-s_0)-\psi^{(m)}(k_2-s_0))$$
which is a rational number. 
\end{proof}

\section{\large{ Hilbert modular forms}}
Let $\F$ be a totally real number field of degree $n$, with ring of integers 
$\cO_\F$, different $\mathfrak{D}_\F$ and discriminant $d_\F$.  For a non-zero 
integral ideal 
$\mathfrak{a}$, let $\N(\mathfrak{a})=[\cO_\F:\mathfrak{a}]$ denote its 
absolute norm.  We consider $\F$ as a subring of $\R^n$ by means of the embedding 
$$
\alpha \mapsto(\eta_1(\alpha), \ldots, \eta_n(\alpha)),
$$ 
where $\eta=(\eta_1, \ldots, \eta_n)$ are the distinct real embeddings of $\F$.
Let $h$ be the narrow class number of $\F$ and $\{t_h\}$ 
be a set of representatives of the narrow class group. 
For a fixed integral ideal $\n$ of $\F$ and $1 \leq \nu \leq h$, we 
define a congruence subgroup $\Gamma_{\nu}(\n)$ of $\GL_2(\F)$ as
$$
\Gamma_{\nu}(\n) =\left\{
\begin{pmatrix}
a &t_\nu^{-1}b\\
t_{\nu}c&d
\end{pmatrix}\Big| 
\begin{array}{lll}
a\in \cO_\F,&b\in \mathfrak{D}_\F^{-1},&\\
c \in \n \mathfrak{D}_\F, &d \in \cO_\F, &ad-bc \in \cO_\F^{\times}
\end{array}
\right\}.
$$
For any integer $k\geq 1$, we use $\kk$ to denote the $n$-tuple 
$(k, \ldots, k)$. Let $f_\nu$ be a Hilbert modular form of weight 
$\kk$ with respect to $\Gamma_\nu(\n)$. Then $f_\nu$ has a Fourier series expansion 
$$f_\nu(z)=\sum_{\xi}a_\nu(\xi)e^{2\pi i\xi z}$$
where $\xi$ runs over totally positive elements in $t_{\nu}\cO_\F$ or $\xi=0$. 
 
A Hilbert modular form $\f=(f_1, \ldots, f_{h})$ of weight $\kk$ and level 
$\n$ is a cusp form if each $f_{\nu}$ is a cusp form of 
weight $\kk$ on $\Gamma_{\nu}(\n)$. 

Let $\gamma=(\gamma_1, \ldots, \gamma_n) \in \GL_2(\R)^n$. Then 
for a Hilbert modular form $\f = (f_1, \ldots, f_{h})$ of weight~$\kk$
and for $z=(z_1, \ldots, z_{n})\in \bH^n$, we use the notations
$$
f_{\nu} |_\kk\gamma(z)= 
\prod_{i} j(\gamma_{i}, z_{i})^{-k}f_{\nu} (\gamma z), 
\qquad f_{\nu} ||_\kk\gamma(z)
= 
\prod_{i}\det{\gamma_{i}}^{k/2} j(\gamma_i, z_i)^{-k}f_{\nu}(\gamma z)
$$
where $1 \le \nu \le h$ and $j \left(
\begin{pmatrix} a &b\\
c&d
\end{pmatrix}, ~z\right) = cz + d.
$
For an integral ideal $\m$ of $\F$, there exists a unique $\nu$ and a 
totally positive $\xi \in t_{\nu}$ such that $\m=\xi t_{\nu}^{-1}\cO_\F$. 
Let $c(\m,\f)=a_{\nu}(\xi)\xi^{-k/2}\N(\m)^{k/2}$. Then the $L$-function 
attached to $\f$ is given by
$$
L(\f,s) = \sum_{ \m}\frac{\mathfrak{c}(\m, \f)}{\N(\m)^s}, \qquad \Re(s)\gg 1,
$$
where $\m$ runs over all non-zero integral ideals of $\cO_\F$. 
It has an analytic continuation to an entire function when $\f$ is a cusp form.\\
Let $\beta_{\nu}=\begin{pmatrix}
0&1\\
-q_{\nu}&0
\end{pmatrix}
$, 
where $q_{\nu}$ is a totally positive element of  $\F$. Then there exists unique 
$\nu'$ such that $t_{\nu}t_{\nu'}\cO_\F\n\mathfrak{D}_\F^2=q_{\nu}\cO_\F$.
Let $f'_{\nu'}=(-1)^kf_{\nu}||_\kk\beta_{\nu}$. Then $\f |J_{\n}=(f'_1, \ldots, f'_{h})$ 
is a Hilbert modular form of weight $\kk$ with respect to $\Gamma_{v'}(\n)$.
The completed $L$-function associated to $\f$ is defined as
\begin{equation}\label{lfunHMF}
 \Lambda(\f, s)=\N(\n\mathfrak{D}_\F^2)^{s/2}(2\pi)^{-ns}\Gamma(s)^nL(\f, s) 
\end{equation}
and satisfies the functional equation (see \cite{GS2})
\begin{equation}\label{FunHMF}
\Lambda(\f, s)=i^{nk}\Lambda(\f|J_\n, k-s),
\end{equation}
where $\f|J_\n= (f'_1, \ldots, f'_{h})$ is as defined above.
We say that $\f$ is normalized if $c(\cO_\F, \f)=1$.
From Proposition 2.10 of \cite{GS2}, we know that if $\f$ is a normalized cuspidal Hilbert Hecke
eigenform of weight $\kk=(k, \ldots, k)$ and conductor $\n$, then 
$\f|J_{\n}=c\f$ for some non-zero constant $c$.  
Recently, Hamieh and Raji \cite{HR} showed that for $k\gg 1$
and for certain points $s_0$ in the critical strip, there exists a Hecke 
eigenform $\f\in \mathcal{S}_{\kk}({\rm SL}_2(\cO_\F))$ 
such that $L(\f, s_0) \neq 0$.
Let $E(N,\kk,s_0) $ be the set of all normalized primitive cuspidal Hilbert 
Hecke eigenforms of weight $\kk$ and level $\n$ such 
that $\N(\n)=N$ and $L(\f, s_0) \neq 0$. Here we have
the following theorem.

\begin{thm}\label{noname}
Let $k \geq 1$ and $s_0 \in \C$. Then atmost one element of the set 
$$
\left\{ \frac{L'(\f_N,s_0)}{L(\f_N,s_0)} +
\frac{L'(\f_N,k-s_0)}{L(\f_N,k-s_0)}
\Big| \f_N\in E(N,\kk,s_0), N \geq 1 \right\}
$$ 
is algebraic. Furthermore, for $k \geq 5$, $n\geq 3$ and 
$s_0= k/2+a/b \in \Q$ with $ 0\leq a/b \le k/2-2$, 
we have
$$
\frac{L'(\f, s_0)}{L(\f, s_0)}+\frac{L'(\f, k- s_0)}{L(\f, k- s_0)} \neq 0.
$$
\end{thm}

\begin{rmk}
The non-vanishing result also holds when $n=2$ and $k \geq 8.$
\end{rmk}
\begin{proof}
Let $\f \in E(N,\kk, s_0)$ 
with weight $\kk$ and level $\n$. Using \eqref{lfunHMF} and \eqref{FunHMF}, we get
$$
\N(\n\mathfrak{D}_\F^2)^{s/2}(2\pi)^{-ns}\Gamma(s)^nL(\f,s)
=
ci^{nk}\N(\n\mathfrak{D}_\F^2)^{(k-s)/2}(2\pi)^{-n(k-s)}
\Gamma(k-s)^nL(\f, k-s).
$$
Taking logarithmic derivative, we have
\begin{equation}\label{derHMF}
\frac12 \log{\N(\n\mathfrak{D}_\F^2)}-n\log{2\pi}+n\psi(s)
+ \frac{L'(\f, s)}{L(\f, s)}=-\frac12 \log{\N(\n\mathfrak{D}_\F^2)} 
+ n\log{2\pi} - n\psi(k-s)-\frac{L'(\f, k-s)}{L(\f, k-s)}.
\end{equation}
Putting $s=s_0$, and using the fact that $\N(\mathfrak{D}_\F)=d_\F$,
we get 
\begin{equation}\label{HMF2}
\frac{L'(\f,s_0)}{L(\f,s_0)}+\frac{L'(\f,k-s_0)}{L(\f,k-s_0)}
=
-\log{N}-2\log{d_\F}+2n\log(2\pi)-n\psi(s_0) -n \psi(k-s_0).
\end{equation}
Let $\f_1 \in E(N_1,\kk, s_0)$ and $\f_2 \in E(N_2,\kk, s_0)$ be 
such that their corresponding values are distinct and algebraic. 
Then by \eqref{HMF2}, the difference of their corresponding values, 
$\log{(N_2/N_1)}$, is
also algebraic. This contradicts Lindemann's theorem.
For the second part, the assumptions on $s_0$ and $k$ 
imply $ k/2+a/b \ge 5/2$ and $k/2 - a/b \ge 2$ and therefore,
\begin{equation}\label{negative}
-\log\mathfrak{N}(\n\mathfrak{D}_\F^2)
-n\psi\left(\frac{k}{2}+\frac{a}{b}\right)-n\psi\left(\frac{k}{2}-\frac{a}{b}\right)
\leq -2\log{d_\F}-n\psi\left(\frac{5}{2}\right)-n\psi(2). 
\end{equation}
Using Minkowski's bound, we have
$\log{d_\F}\geq 2n\log{n}-2\log{n!}$. Putting this in the above
equation and using \eqref{HMF2}, we get
$$
\frac{L'(\f,s_0)}{L(\f,s_0)}+\frac{L'(\f,k-s_0)}{L(\f,k-s_0)}
\leq 
-2n\log{n}+2\log{n!}+2n\log{2\pi}
-n\psi\left(\frac{5}{2}\right)-n\psi(2)
$$
which is negative for $n \geq 5$. 
For $n=3$ (resp. $4$), following Tanabe \cite{NT}, we use the minimal discriminant 
of extensions of degree $3$ (resp. $4$) and show that 
R.H.S. of \eqref{negative} is negative to conclude the result.
\end{proof}

\begin{thm}
Let $k\geq 1$ and $s_0 \in \C$.
Let $J$ be a non-empty, finite set of positive 
integers $N$ such that $(N,d_\F)=1$ and 
 $E(N,\kk, s_0)$ is non-empty. Further assume that $J$ has property $\bf{A}$. Then 
$$
\left\{ \frac{L'(\f_N,s_0)}{L(\f_N,s_0)} +
\frac{L'(\f_N,k-s_0)}{L(\f_N,k-s_0)}
 \Big| \f_N\in E(N,\kk,s_0), N \in J \right\}
 $$ 
spans a vector space over $\overline{\Q}$ of dimension at least $|J|-1$.
\end{thm}

\begin{proof}
Since the set $J$ has property $\bf{A}$, the set $\{\log{N}|N \in J\}$
is linearly independent over $\Q$. By Baker's theorem this set in
linearly independent over $\overline{\Q}$. The result now follows 
using \eqref{HMF2} and \lemref{LI}.
\end{proof}

\begin{thm}
Let $k \geq 2$. Then there exists an integer $q_0$ such that for any
integer $q \ge 7$ coprime to $q_0$ and 
$\f \in \cap_{\substack{1\leq a \leq q\\ (a,q)=1}}E(N, \kk, a/q)$, 
the $\Q$-vector space spanned by the set
$$
\left\{ \frac{L'(\f,a/q)}{L(\f,a/q)} 
+\frac{L'(\f,k-a/q)}{L(\f,k-a/q)}
~\Big|~
 1 \leq a  <q/2, (a,q)=1 \right\}
$$
has dimension at least 
$\frac{\phi(q)}{2}-2$. 
\end{thm}
\begin{proof}
Putting $s_0=a/q$ in \eqref{HMF2}, we get
$$
\frac{L'(\f,a/q)}{L(\f,a/q)}+\frac{L'(\f,k-a/q)}{L(\f,k-a/q)}
=
\log{\frac{(2\pi)^{2n}}{Nd_\F^2}}-n\psi(a/q)-n \psi(1-a/q)-nr(k,a,q),
$$
where $\displaystyle r(k,a,q)= \sum_{j=1}^{k-1}\frac{1}{j-a/q} \in \Q$. 
The result now follows by \rmkref{dimension} and \lemref{LI}.
\end{proof}

\subsection{Higher derivatives of L-functions associated to Hilbert Modular Forms}

For any integer $m\geq 1$, taking the $m$-th derivative of \eqref{derHMF}, we get 
\begin{equation}\label{highHMF}
\left(\frac{L'(\f,s)}{L(\f,s)}\right)^{(m)}+ 
(-1)^m \left(\frac{L'(\f,k-s)}{L(\f,k-s)}\right)^{(m)}=
-n(\psi^{(m)}(s)+(-1)^m\psi^{(m)}(k-s)).
\end{equation}
\begin{thm}
Let $M , k\geq 1$ be integers. Consider the set 
$$
\mathcal{L}_{2,M}=\left\{\frac{L^{(m)}(\f,k/2)}{L(\f,k/2)}
~\Big|~
N \geq 1, \f \in E(N,\kk, k/2), 1 \leq m \leq M\right\}.
$$
Then we have 
$$
\tr_{\Q}(\mathcal{L}_{2,M}) 
~\geq~  
\tr_{\Q}\left(\left\{\psi^{(2m)}(1) ~\Big| ~1 \leq m \leq \left[\frac{M-1}{2}\right]
\right\}\right).
$$
\end{thm}

\begin{proof} 
Since $\f \in E(N,\kk, k/2)$, for $s=k/2$, using \eqref{highHMF}, we get  
$$
\left(\frac{L'(\f, s)}{L(\f, s)}\right)^{(2m)}|_{s=k/2} 
=
-\psi^{(2m)}\left(\frac{k}{2}\right).
$$
Let 
$$
\tilde{\mathcal{L}}_{2,M}
=\left\{\left(\frac{L'(\f, s)}{L(\f, s)}\right)^{(m)}|_{s=k/2} 
~\middle\vert~ 
N \geq 1, \f \in E(N,\kk, k/2), 0 \leq m \leq M-1 \right\}.
$$
Using induction, we see that for any  $j \in \Z_{\ge 1}$,  
$\displaystyle\frac{L^{(j)}(\f, s)}{L(\f, s)}$ can be expressed as a 
polynomial in $\displaystyle\frac{L'(\f, s)}{L(\f, s)}$, 
$\displaystyle\Big(\frac{L'(\f, s)}{L(\f, s)}\Big)^{(1)}, \ldots ,$ 
$\displaystyle\Big(\frac{L'(\f, s)}{L(\f, s)}\Big)^{(j-1)}$ 
with coefficients in $\Z$ and conversely 
for any ${j \in \bN}$, 
$\displaystyle\Big(\frac{L'(\f, s)}{L(\f, s)}\Big)^{(j)}$
can be expressed as a polynomial in
$\displaystyle\frac{L'(\f, s)}{L(\f, s)}, 
\frac{L^{(2)}(\f, s)}{L(\f, s)}, \ldots, 
\frac{L^{(j+1)}(\f, s)}{L(\f, s)}$
with coefficients in $\Z$. This implies 
$\tr_{\Q}(\mathcal{L}_{2,M})= \tr_{\Q}(\tilde{\mathcal{L}}_{2,M})$. 
Using this observation along with \eqref{Gamma} and \eqref{psi12}
we get the result.
\end{proof}

\begin{thm}
For $s_0 \in \Q$, let
$$
\mathcal{S}(s_0)=\{\f ~|~ \f \in E(N,\kk, s_0) {\rm{~for~  integers~}} N, k \geq 1\}.
$$
Then for any $m \geq 1$, if at least one element of the set
$$
\left\{\left(\left(\frac{L'(\f, s)}{L(\f, s)}\right)^{(m)} 
+ 
(-1)^m \left(\frac{L'(\f, k-s)}{L(\f, k-s)}\right)^{(m)}\right)|_{s=s_0}
~\Big|~ \f \in \mathcal{S}(s_0) \right\}
$$
is algebraic, then all the elements of this set are algebraic.
\end{thm}

\begin{proof}
Let $\f_1, \f_2\in \mathcal{S}(s_0)$. By \eqref{highHMF}, we know that 
the result is true when $\kk_1=\kk_2$. For $\kk_1\neq \kk_2$, again 
using \eqref{highHMF}, we note that the difference of their corresponding values is 
$$
(-1)^mn(\psi^{(m)}(k_1-s_0)-\psi^{(m)}(k_2-s_0))
$$
which is a rational number. 
\end{proof}

\section{\large{ Siegel modular forms}}

For an integer $g \geq 2$, let ${\rm Sp}_g(\Z)$
denote the full Siegel modular group of genus $g$ and
$\Lambda_g^+$ the set of positive definite half-integral
matrices of size $g$. For an even integer $k$, 
let $f$ be a Siegel cusp form of weight $k$ with respect 
to ${\rm Sp}_g(\Z)$ with Fourier coefficients $a_f(T)$, 
where $T$ is a symmetric positive definite 
half-integral matrix of size $g$. 
The Koecher-Maass series attached to $f$ is defined by
$$
D_f(s)= \sum_{T\in \Lambda_g^+/{\rm GL}_g(\Z)}
\frac{a(T)}{\epsilon(T)}(\det T)^{-s}, \qquad
\Re(s) \gg 1 
$$
where the sum is over elements of $\Lambda_g^+$ 
modulo the right action of ${\rm GL} _g(\Z)$ 
on $\Lambda_g^+$ given by  $T \mapsto U^tTU$
and  $\epsilon(T)=\# \{U \in {\rm GL}_g(\Z) \vert U^tTU=T\}$.
Here $U^t$ denotes transpose of $U$.
The completed Koecher-Maass series is defined by
\begin{equation}\label{compSMF}
D_f^*(s)=(2\pi)^{-gs}\prod_{\nu=0}^{g-1}\Gamma
\left(s-\frac{\nu}{2}\right)D_f(s)
\end{equation} 
and has a holomorphic continuation to $\C$. 
It satisfies the functional equation
\begin{equation}\label{funSMF}
D_f^*(k-s)=(-1)^{\frac{gk}{2}}D_f^*(s).
\end{equation}
For more details, see \cite{DK, KS, HM}. In this section, we shall always assume
that $k>2(g+1)$.
Let $E(g,k,s_0)$ be the set of all Siegel Hecke eigen cusp form
$f$ of weight $k$ with respect to ${\rm Sp}_g(\Z)$ such that
$D_f(s_0)\neq 0$.
We know from the works of \cite{DK,WK2}) that for 
$k \gg 1$ and for certain points $s_0$ in the critical strip, $E(g,k, s_0)$ is non-empty.
In this set-up, we have the following theorem.

\begin{thm}
Let $k>2(g+1)$ be a fixed integer and $s_0 \in \C$. 
For $f\in E(g,k,s_0)$, we have 
\begin{equation}\label{SMF}
\frac{D'_f(s_0)}{D_f(s_0)}+\frac{D'_f(k-s_0)}{D_f(k-s_0)}
=2g\log{2\pi}-\sum_{\nu=0}^{g-1}
\psi\left(s_0-\frac{\nu}{2}\right)-\sum_{\nu=0}^{g-1}
\psi\left(k-s_0-\frac{\nu}{2}\right).
\end{equation}
Furthermore, if $k\geq 2(g+7)$, for any  $s_0= k/2+a/b \in \Q$ 
with $ 0\leq a/b <(g+1)/2$, 
we have
$$
\frac{D'_f(s_0)}{D_f(s_0)}+\frac{D'_f(k-s_0)}{D_f(k-s_0)}\neq 0.
$$
\end{thm}

\begin{proof}
Using \eqref{compSMF} and \eqref{funSMF}, we get
$$
(2\pi)^{-g(k-s)}\prod_{\nu=0}^{g-1}\Gamma
\left(k-s-\frac{\nu}{2}\right)D_f(k-s)=(-1)^{\frac{gk}{2}}
(2\pi)^{-gs}\prod_{\nu=0}^{g-1}\Gamma
\left(s-\frac{\nu}{2}\right)D_f(s).
$$
Taking logarithmic derivative, we get
\begin{equation}\label{derSMF}
g\log{2\pi}-\sum_{\nu=0}^{g-1}\psi\left(k-s-\frac{\nu}{2}\right)
-\frac{D'_f(k-s)}{D_f(k-s)}=-g\log{2\pi}+\sum_{\nu=0}^{g-1}
\psi\left(s-\frac{\nu}{2}\right)+\frac{D'_f(s)}{D_f(s)}.
\end{equation}
Putting $s=s_0$, we get the first part of the result.
For $s_0$ and $k$ as in the theorem, we have $k/2+a/b \ge g + 7$ 
and $k/2-a/b \ge (g-1)/2+ 7$. Therefore, trivially bounding the 
right hand side of \eqref{SMF}, we get 
$$
\frac{D'_f(s_0)}{D_f(s_0)}+\frac{D'_f(k-s_0)}{D_f(k-s_0)}
<2g\log{2\pi}-2g\psi(7)<0.
$$
\end{proof}

\begin{thm}
Let $g \geq 2$ be even. Then there exists an integer $q_0$ such that
for any odd integer $q \ge 7$ coprime to $q_0$ and 
$f \in \cap_{\substack{1\leq a \leq q\\ (a,q)=1}}E(g, k, a/q)$, the $\Q$-vector 
space spanned by the set
$$
\left\{ \frac{D'_f(a/q)}{D_f(a/q)}+
\frac{D'_f(k-a/q)}{D_f(k-a/q)}
~\Big|~
 1 \leq a  <q/2, (a,q)=1 \right\}
$$
has dimension at least 
$\frac{\phi(q)}{2}-2$. 
\end{thm}
\begin{proof}
Putting $s_0=a/q$ in \eqref{SMF}, we get
\begin{eqnarray*}
\frac{D'_f(a/q)}{D_f(a/q)}+\frac{D'_f(k-a/q)}{D_f(k-a/q)}
&=& 2 g \log{2\pi} - \sum_{\nu=0}^{g-1}
\psi\left(\frac{a}{q}-\frac{\nu}{2}\right) - \sum_{\nu=0}^{g-1}
\psi\left(k-\frac{a}{q}-\frac{\nu}{2}\right)\\
&=&2 g \log{2\pi}-\frac{g}{2}\left(\psi\left(\frac{a}{q}\right)
+\psi\left(\frac{a}{q}-\frac12 \right)\right)\\
&&-\frac{g}{2}\left( \psi\left(1-\frac{a}{q}\right) +
\psi\left(\frac12-\frac{a}{q} \right) \right) -r(g,k,a,q),
\end{eqnarray*}
where $r(g, k, a, q) \in \Q$. For every $1 \leq a <q/2$ with $(a,q)=1$, there exists a
unique $1 \leq b <q/2$ with $(b,q)=1$ such that
\begin{eqnarray*}
\frac{D'_f(a/q)}{D_f(a/q)}+\frac{D'_f(k-a/q)}{D_f(k-a/q)}
&=& 2 g \log{2\pi}-g\psi(b/q)-g\psi(1-b/q)-2g\log{2}-r^*(g,k,a,q)\\
&=& 2 g \log{\pi}-g\psi(b/q)-g\psi(1-b/q)-r^*(g,k,a,q)
\end{eqnarray*}
for some $r^*(g, k, a,q) \in \Q$.
The result now follows from \rmkref{dimension} and \lemref{LI}.
\end{proof}

\subsection{Higher derivatives of Koecher-Maass series associated to Siegel Modular Forms} 

For any integer $m\geq 1$, taking the $m$-th derivative 
of \eqref{derSMF}, we get 
\begin{equation}\label{highSMF}
\left(\frac{D'_f(s)}{D_f(s)}\right)^{(m)}
+(-1)^m\left(\frac{D'_f(k-s)}{D_f(k-s)}\right)^{(m)}
=-\sum_{\nu=0}^{g-1}\psi^{(m)}\left(s-\frac{\nu}{2}\right)
+(-1)^{m+1}\sum_{\nu=0}^{g-1}\psi^{(m)}\left(k-s-\frac{\nu}{2}\right).
\end{equation}
\begin{thm}
Let $M \geq 1$. Consider the set 
$$
\mathcal{L}_{M}=\left\{\frac{D^{(m)}_f(k/2)}{D_f(k/2)} ~\middle 
\vert~ g\geq 1, f \in E(g,k,k/2), 1 \leq m \leq M\right\}.
$$
Then we have 
$$
\tr_{\Q}(\mathcal{L}_{M})
~\geq~  
\tr_{\Q}\left(\left\{\psi^{(2m)}(1) ~\Big| 
~1 \leq m \leq \left[\frac{M-1}{2}\right]\right\}\right).
$$
\end{thm}
\begin{proof}
Since $f \in E(g, k, k/2)$, for $s=k/2$, using \eqref{highSMF}, 
we get 
$$
\left(\frac{D'_f(s)}{D_f(s)}\right)^{(2m)}|_{s=k/2}=
-\sum_{\nu=0}^{g-1}\psi^{(2m)}\left(\frac{k}{2}-\frac{\nu}{2}\right).
$$
Let 
$$
\tilde{\mathcal{L}}_{M}
=\left\{\left(\frac{D'_f(s)}{D_f(s)}\right)^{(m)} ~\middle \vert~
g\geq 1, f \in E(g,k,k/2), 0 \leq m \leq M-1\right\}.
$$
Using induction, we observe that for any $j \in \Z_{\ge 1}$,  
$\displaystyle\frac{D^{(j)}_f(s)}{D_f(s)}$ can be expressed as a 
polynomial in $\displaystyle\frac{D'_f(s)}{D_f(s)}$, 
$\displaystyle\Big(\frac{D'_f(s)}{D_f(s)}\Big)^{(1)}, \ldots ,$ 
$\displaystyle\Big(\frac{D'_f(s)}{D_f(s)}\Big)^{(j-1)}$ 
with coefficients in $\Z$ and conversely 
for any ${j \in \bN}$, 
$\displaystyle\Big(\frac{D'_f(s)}{D_f(s)}\Big)^{(j)}$
can be expressed as a polynomial in
$\displaystyle\frac{D'_f(s)}{D_f(s)}, 
\frac{D^{(j)}_f(s)}{D_f(s)}, \ldots, 
\frac{D^{(j+1)}_f(s)}{D_f(s)}$
with coefficients in $\Z$.
This implies 
$\tr_{\Q}(\mathcal{L}_{M})= \tr_{\Q}(\tilde{\mathcal{L}}_{M})$. 
Using this observation along with \eqref{Gamma} and \eqref{psi12},  
we get the result.
\end{proof}

\begin{thm}
For $s_0\in \Q$, let
$$
\mathcal{S}(s_0)=\{f ~|~ f \in E(g,k,s_0) 
{\rm{~for~  integers~}} g, k \geq 1\}.
$$
Then for any $m \geq 1$, if at least one element of the set
$$
\left\{\left(\left(\frac{D'_f(s)}{D_f(s)}\right)^{(m)}
+(-1)^m\left(\frac{D'_f(k-s)}{D_f(k-s)}\right)^{(m)}\right)|_{s=s_0}
~\Big|~ f \in \mathcal{S}(s_0) \right\}
$$
is algebraic, then all the elements of this set are algebraic.
\end{thm}

\begin{proof}
Let $f_1, f_2\in \mathcal{S}(s_0)$. By \eqref{highSMF}, we know that 
the result is true when $k_1=k_2$. For $k_1\neq k_2$, again 
using \eqref{highSMF}, we note that the difference of their corresponding values is 
$$
(-1)^{m}\sum_{\nu=0}^{g-1}\left(\psi^{(m)}\left(k_1-s_0-\frac{\nu}{2}\right)
-\psi^{(m)}\left(k_2-s_0-\frac{\nu}{2}\right)\right).
$$
which is a rational number. 
\end{proof}

\section{Acknowledgements}
The author would like to thank Professor Sanoli Gun for suggesting the problem and for her guidance throughout the paper and Professor Purusottam Rath for helpful discussions in improving the paper. The author would like to thank DAE Number Theory Plan Project.
The author would also like to thank the referee for careful reading of the paper and kind suggestions.

\end{document}